\documentclass[a4paper,12pt]{article}
\usepackage[top=2.5cm, bottom=2.5cm, left=2.5cm, right=2.5cm]{geometry}

\usepackage{graphicx}
\usepackage{amssymb}
\usepackage{amsthm}
\usepackage{amsmath}
\usepackage{authblk}
\usepackage{xcolor}
\usepackage{enumitem}
\usepackage[english]{babel}
\usepackage[utf8]{inputenc}  
\usepackage[T1]{fontenc}
\usepackage[hidelinks]{hyperref}
\usepackage{lipsum}

\newtheorem{theorem}{Theorem}[section]

\newtheorem{theoremA}{Theorem}

\newtheorem{theoremB}{Theorem}

\newtheorem{lemma}{Lemma}[section]

\newtheorem{proposition}{Proposition}[section]

\newtheorem{corollary}{Corollary}[section]

\newtheorem{definition}{Definition}[section]

\usepackage{titlesec}

\titleformat{\section}
  {\centering\normalfont\Large}
  {\thesection.}{0.5em}{}

\titleformat{\subsection}
  {\centering\normalfont\large}
  {\thesubsection.}{0.5em}{}

\titleformat{\subsubsection}
  {\centering\normalfont\normalsize}
  {\thesubsubsection.}{0.5em}{}

\title{Extension of hypercyclic and frequently hypercyclic subspaces.}

\author[1]{Felipe Carvalho Silva \thanks{\texttt{felipe.silva@ime.unicamp.br}}}
\author[2]{Geivison Ribeiro \thanks{\texttt{geivison@unicamp.br and geivison.ribeiro@academico.ufpb.br}}}
\author[3]{Régis Varão \thanks{ \texttt{varao@unicamp.br}}}

\affil[1,2,3]{\footnotesize Departamento de Matemática, Instituto de Matemática, Estatística e Computação Científica, Universidade Estadual de Campinas, Rua Sérgio Buarque de Holanda, 651, 13083-970, Campinas-SP, Brazil.}

\date{}

\begin{document}
\maketitle
\noindent\rule{\linewidth}{0.4pt}

\begin{abstract} 
\noindent 
We focus on the existence of large linear structures within the sets of hypercyclic and frequently hypercyclic vectors. For operators $T$ satisfying Kitai's Criterion or the Frequent Hypercyclicity Criterion, we analyze the fundamental linear space $\{f(T)x |  f \in H(\mathbb{C})\}$, studied  by Herrero, Bourdon, Bès, Wengenroth, and many others. We show that the set $\{f(T)x |  f \in H(\mathbb{C})\}$ can be extended within $HC(T) \cup \{0\}$ or $FHC(T) \cup \{0\}$ if $x \in HC(T)$ or $x \in FHC(T)$, respectively. The extension is such that the quotient of the new space with $\{ f(T)x \mid f \in H(\mathbb{C}) \}$ has dimension $\mathfrak{c}$ (the cardinality of the continuum). Second, we prove that generically a finite-dimensional subspace contained in $HC(T) \cup \{0\}$ can be enlarged to a subspace of dimension $\mathfrak{c}$. Third, we establish sufficient conditions for extending arbitrary linear subspaces both from  $HC(T) \cup \{0\}$ and $FHC(T) \cup \{0\}$ to larger subspaces of dimension $\mathfrak{c}$.
\end{abstract}

\noindent\rule{\linewidth}{0.4pt}

\section{Introduction}

The search of linearity on a set which do not have a priori a linear structure is an area of research known as lineability, spaceability, and related notions. This theory has found several applications; see, for instance, \cite{Bernal.et.al.JFA,botelho.diniz.pellegrino,Cariello-favaro-seoane,cariello-seoane-juan,pellegrino-teixeira}. We also refer the reader to the Aron, Bernal-Gonzalez, Pellegrino and Sepulveda's monograph \cite{aron2015lineability} for more references and a comprehensive overview of the subject. In the present work, we focus on the problem of lineability for the sets of hypercyclic and frequently hypercyclic vectors.

The study of linear structures within the set of hypercyclic vectors has a long history, with many contributions over the years; see, for example, \cite{herrero1991limits, bourdon1993invariant, bernal2000hypercyclic, leon2001spectral, wengenroth2003hypercyclic, bes2006hypercyclic}. For a panoramic view of the early developments in the area, we refer the reader to the survey by Grosse-Erdmann \cite{grosse-erdmann1999universal}.

Among others, Herrero \cite{herrero1991limits}, Bourdon \cite{bourdon1993invariant}, Bès \cite{bes2006hypercyclic} and Wengenroth \cite{wengenroth2003hypercyclic} studied the density of the set $\{p(T)x | \text{ p is polynomial}\}$ of hypercyclic vectors.
In our paper, we investigate related sets from the perspective of constructing invariant linear structures contained in the set of hypercyclic vectors.

Recall that a continuous operator $T:X \rightarrow X$ on a Banach space is called hypercyclic if there exists a vector $x \in X$ whose orbit is dense in $X$. Such a vector is called a hypercyclic vector, and the set of all hypercyclic vectors of $T$ is denoted by $HC(T)$.

It is well known that if $T$ is hypercyclic on a Banach space, then one can construct a dense linear subspace $W \subset HC(T) \cup {0}$ with $\dim(W)=\mathfrak{c}$. Indeed, given $x \in HC(T)$, one may consider
\[
W:=\{f(T)x \mid f \in H(\mathbb{C})\}, 
\]
where $H(\mathbb{C})$ is the set of entire functions. For details, see Section 4.3 in \cite{aron2015lineability}. As mentioned above, this construction has been extensively studied, and it is natural to ask whether the set $\{ f(T)x \mid f \in H(\mathbb{C}) \}$ is maximal with respect to inclusion in $HC(T) \cup \{0\}$.

Our first result shows that, in many situations, this is not the case. The problem of extending a given linear set differs in nature from that of merely finding linear structures inside a set: one must construct a larger space while preserving the properties of the original one.

Rather than working with arbitrary hypercyclic operators, we will assume that $T$ satisfies the Kitai's Criterion (see Theorem \ref{thm:Kitai} in \S \ref{Sec:Preliminaries}). 

\begin{theoremA} \label{thm:ext_f(T)} 
For any operator $T$ satisfying Kitai's Criterion on a separable Banach space $X$ and any hypercyclic vector $x$, the set $\{ f(T)x \mid f \in H(\mathbb{C}) \}$ can be extended in $HC(T) \cup \{0\}$. Moreover, the extended set can be taken such that its quotient with $\{ f(T)x \mid f \in H(\mathbb{C}) \}$ has dimension $\mathfrak{c}$.
\end{theoremA}

For more on the extended set from the result above see Corollary \ref{cor:sum_entire_function} after the proof of Theorem \ref{thm:ext_f(T)} at the end of section \S \ref{Sec:Proof_ext_f(T)}.

The problem of extending a given subspace $W \subset HC(T) \cup \{0\}$  to a larger subspace still contained in $HC(T) \cup \{0\}$ remains open in general. In this work, we make progress in this direction by showing that a broad class of such subspaces can indeed be extended.

\begin{theoremA}\label{thm:G_delta_kitai}
Let $T$ be an operator satisfying Kitai's Criterion on a separable Banach space $X$.
\begin{enumerate}[label=\roman*)]
\item If $n \in \mathbb N$, then there exists a $G_\delta$-dense set $\mathcal{R}_{n} \subset X^{n}$ such that for any $(y_1, \ldots, y_n) \in \mathcal{R}_{n}$ the set $span( \{y_1, \ldots, y_n \})$ is contained $HC(T)\cup \{0\}$ and can be extended inside $HC(T)\cup \{0\}$ to a subspace of dimension $\mathfrak{c}$.

\item There exists a Mycielski set  $\mathcal{R}_{\mathbb{N}} \subset X$, such that for any $\{y_i \mid i \in \mathbb N\} \subset \mathcal{R}_{\mathbb{N}}$ the set $span(\{y_i \mid i \in \mathbb N\})$ is contained $ HC(T)\cup \{0\}$ and can be extended inside $ HC(T)\cup \{0\}$ to a subspace of dimension $\mathfrak{c}$.
\end{enumerate}
\end{theoremA}

Where $X^{n}$ denote the product space of $X$ endowed with the product topology and $span(A)$ is the set of all finite linear combinations of elements of $A$. The set $\mathcal{R}_{\mathbb{N}}$ appearing in the statement of the above result is dense in $X$. In fact, for every open set $U \subset X$, the intersection $\mathcal{R}_{\mathbb{N}} \cap U$ contains a Cantor set. Furthermore, the sets $\mathcal{R}_{\mathbb{N}}$ can be chosen to satisfy, for any $n \in \mathbb{N}$,
\[
y_{1}, \dots , y_{n} \in \mathcal{R}_{\mathbb{N}} \text{ pairwise distinct }  \Rightarrow (y_{1} , \dots , y_{n}) \in \mathcal{R}_{n}.
\]
See Theorem \ref{thm:mycielski} on  Mycielski's set. The next technical result is needed in order to prove Theorem \ref{thm:G_delta_kitai}.

\begin{theoremA} \label{thm:ext_Kitai}
Let $T:X \rightarrow X$ be an operator satisfying Kitai's Criterion  on a separable Banach space $X$. Given a countable set $\alpha \subset \mathbb{N}$, finite or infinite, and a hypercyclic subspace $span(\{y_i \mid i \in \alpha \}) \subset HC(T) \cup \{0\}$ that admits a sequence $(n_k)$ such that $\sup_{i \in F} \|T^{n_k} y_i\| \xrightarrow[k \to \infty]{} 0$ for every finite subset $F$ of $\alpha$, then there exists a hypercyclic subspace $W$ such that $span(\{y_i \mid i \in \alpha \}) \subset W \subset HC(T) \cup \{0\}$ and $\dim W = \mathfrak{c}$.
\end{theoremA}

We now turn our attention to the set of frequently hypercyclic vectors. An operator $T$ is said to be frequently hypercyclic if there exists a vector $x \in X$ whose orbit visits every nonempty open set with positive lower density. We denote by $FHC(T)$ the set of frequently hypercyclic vector for $T$ (see Definition \ref{def:fhc-T_and_vectors}). 

The lineability of $FHC(T)$, has been less explored, partly because $FHC(T)$ is a much smaller set than $HC(T)$. Indeed, while $HC(T)$ is a dense $G_\delta$ set by Birkhoff's Transitivity Theorem, the set $FHC(T)$ does not need to share this property (See Proposition 9.19 of \cite{grosseerdmann2011linear}). Consequently, an analogue of Theorem \ref{thm:G_delta_kitai} cannot, in general, be expected in this setting.

Nevertheless, we are able to obtain results analogous to Theorems \ref{thm:ext_f(T)} and \ref{thm:ext_Kitai} for the class of frequently hypercyclic operators.

\begin{theoremB} \label{thm:ext_f(T)_Freq}
For any operator $T$ satisfying the Frequently Hypercyclicity Criterion on a separable Banach space $X$ and any frequently hypercyclic vector $x$ such that $T^{n_{k}}x \rightarrow 0$ for some increasing sequence $(n_{k})$ with positive lower density, the set $\{ f(T)x \mid f \in H(\mathbb{C}) \}$ can be extended in $FHC(T) \cup \{ 0 \}$. Moreover, the extended set can be taken such that its quotient with $\{ f(T)x \mid f \in H(\mathbb{C}) \}$ has dimension $\mathfrak{c}$.
\end{theoremB}

More on the extended set from the result above see Corollary \ref{cor:sum_entire_function_freq} after the proof of Theorem \ref{thm:ext_f(T)_Freq} at the end of section \S \ref{Sec:Proof_ext_f(T)_Freq}. For the Frequent Hypercyclicity Criterion see Theorem \ref{thm:CHF}. 

A careful reader will notice that the proof of the Frequent Hypercyclicity Criterion  and both proofs of Theorem \ref{thm:ext_f(T)_Freq} and \ref{thm:ext_CHF} involve the construction of vectors satisfying the property of (uniform) convergence to zero under some sequence of positive lower density. Therefore, this is a condition that we already know can be fulfilled for our final result.

\begin{theoremB} \label{thm:ext_CHF}
Let $X$ be a separable Banach space and let $T:X \to X$ be an operator satisfying the Frequent Hypercyclicity Criterion.
Given a countable set $\alpha$, finite or infinite, and a hypercyclic subspace $span(\{y_i \mid i \in \alpha \}) \subset FHC(T) \cup \{0\}$ that admits an increasing sequence $(n_{k})$ with positive lower density and $\sup_{i \in F} \|T^{n_k} y_i\| \xrightarrow[k \to \infty]{} 0$ for every finite subset $F$ of $\alpha$, then there exists a frequently hypercyclic subspace $W$ such that $span(\{y_i \mid i \in \alpha \}) \subset W \subset FHC(T) \cup \{0\}$ and $\dim W = \mathfrak{c}$.
\end{theoremB}

\section{Preliminaries} \label{Sec:Preliminaries}
We state some results and definitions needed for our results.

\begin{theorem}[Kitai's Criterion] \label{thm:Kitai}
Let $T:X \rightarrow X$ be an operator on a separable Banach space $X$. If there is a dense subset $X_{0} \subset X$ and a map $S: X_{0} \rightarrow X_{0}$ such that, for any $x \in X_{0}$,

\begin{enumerate}
\item $T^{n} x \rightarrow 0$;    
\item $S^{n} x \rightarrow 0$;
\item $TS x = x$.
\end{enumerate}
Then $T$ is mixing, and in particular hypercyclic.
\end{theorem}

We say the operator $T$ is weakly mixing if $T\oplus T$ is hypercyclic. 
For a set $A \subset \mathbb{N}$, its frequency (or lower density) is defined by 
\[
\text{freq}(A) = \liminf _{N \rightarrow \infty} \frac{\# \lbrace 0 \leq n \leq N \mid n \in A \rbrace}{N+1}.
\]
\begin{definition}\label{def:fhc-T_and_vectors}
An operator $T:X \rightarrow X$ is said to be frequently hypercyclic if there exist $x \in X$ such that, for every nonempty set $U \subset X$,

\[
\liminf _{N \rightarrow \infty} \frac{\# \lbrace 0 \leq n \leq N \mid T^{n} x \in U \rbrace}{N+1} > 0.
\]

In this case, $x$ is called a frequently hypercyclic vector and the set of such vectors are denoted by $FHC(T)$.
\end{definition}

Equivalently, $T:X \rightarrow X$ is frequently hypercyclic if and only if there exists $x \in X$ such that, for every open set $U \subset X$, there exists an increasing sequence $(n_{k})$ of order $k$ such that
\[
T^{n_{k}}x \in U \text{ for all } k \in \mathbb{N}.
\]
Here, the sequence $(n_{k})$ is said to be of order $k$, and we denote $(n_{k}) = \mathcal{O}(k)$,  if the sequence $\left( \frac{n_{k}}{k} \right)$ is bounded. In this case, the frequency with which $x$ visits $U$ is 
\[
\liminf_{n \rightarrow \infty} \frac{k}{n_{k}}.
\]

\begin{lemma}\label{lem:positive_density_subsequence}
Let $(n_{k}) \subset \mathbb{N}$ and $(m_{k}) \subset \mathbb{N}$ be two sequences with positive lower density. Then, $(m_{n_{k}})$ has positive lower density and
\[
\text{freq}[(m_{n_{k}})] \geq \text{freq}[(m_{k})] \cdot \text{freq}[(n_{k})]
\]
\end{lemma}

\begin{proof}
Due to the basic properties of $\liminf$, we have that
\begin{align*}
\text{freq}[(m_{n_{k}})] 
&= \liminf _{k \rightarrow \infty} \frac{k}{m_{n_{k}}} \\
&= \liminf _{k \rightarrow \infty} \frac{k}{n_{k}} \cdot \frac{n_{k}}{m_{n_{k}}} \\
&\geq \left( \liminf _{k \rightarrow \infty} \frac{k}{n_{k}} \right) \cdot \left( \liminf _{k \rightarrow \infty} \frac{n_{k}}{m_{n_{k}}} \right) \\
&= \text{freq}[(m_{k})] \cdot \text{freq}[(n_{k})]
\end{align*}
\end{proof}

\begin{lemma} \label{lem:positive_density}
There exist pairwise disjoint subsets $A(i,\nu)$, $i,\nu \geq 1$, of $\mathbb{N}_0$ 
with positive lower density such that, for any $n \in A(j,\nu)$ and 
$m \in A(k,\mu)$, we have that $n \geq \nu$ and
\[
|n-m| \geq \nu + \mu \quad \text{if } n \neq m.
\]
\end{lemma}

The previous lemma is stated in its classical form, which is commonly used in the construction of a frequently hypercyclic vector and will be relevant for the proof of Theorem \ref{thm:ext_f(T)_Freq}. However, in the proof of Theorem \ref{thm:ext_CHF}, infinitely many frequently hypercyclic vectors will be constructed. For this reason, in order to simplify the notation, we rewrite the lemma in a more convenient form. Note that the following version amounts to an enumeration of the product $\mathbb{N} \times \mathbb{N}$.

\begin{lemma}\label{lem:positive_density_alt}
There exist pairwise disjoint subsets $A(i,j,\nu)$, $i,j,\nu \geq 1$, of $\mathbb{N}_0$ 
with positive lower density such that, for any $n \in A(i,j,\nu)$ and 
$m \in A(k,l,\mu)$, we have that $n \geq \nu$ and
\[
|n-m| \geq \nu + \mu \quad \text{if } n \neq m.
\]
\end{lemma}

\begin{theorem}[Frequent Hypercyclicity Criterion] \label{thm:CHF}
Let $T: X \rightarrow X$ be an operator on a separable Banach space $X$. Suppose that there exists a dense subset $X_{0} \subset X$ and a map
$S: X_{0} \to X_{0}$ such that, for every $x \in X_{0}$,

\begin{enumerate}[label= (\roman*)]
\item \label{CHF(1)} the series $\sum_{n=0}^{\infty} T^{n} x$ converges unconditionally;    
\item \label{CHF(2)} the series $\sum_{n=0}^{\infty} S^{n} x$ converges unconditionally;
\item \label{CHF(3)} $TS x = x$.
\end{enumerate}
Then $T$ is frequently hypercyclic.
\end{theorem}

\begin{definition}
Let $X$ be a Banach space. A sequence $(x_{n})$ in $X$ is said to be $\ell_{\infty}$-independent if for every sequence $(\lambda_{n}) \in \ell_{\infty}$ such that $\sum_{n=1}^{\infty} \lambda_{n} x_{n} = 0$, we have $(\lambda_{n}) = 0$.
\end{definition}

The next result, although already known in the literature we give a proof to make it accessible and also because it will be particularly important for us. 

\begin{theorem} \label{thm:ell_infty-independence}
Let $X$ be a Banach space and let $(x_{n})$ be a linearly independent sequence in $X$. Then there exist scalars $0 < \gamma_{n} \leq 1$ such that $\{\gamma_{n} x_{n} \mid n \in \mathbb{N} \}$ is $\ell_{\infty}$-independent. 
\end{theorem}

\begin{proof}
We construct these scalars inductively. Set $\gamma_{1} = 1$ and suppose that $\gamma_{1}, \dots, \gamma_{n-1}$ have been defined and are all nonzero. Now, consider the set 
\[
K_{n-1} = \left\{ \sum_{i=1}^{n-1} \lambda_{i} \gamma_{i} x_{i}  : |\lambda_{i}| \leq 1 \ \text{for} \ i = 1,\dots,n-1 \ \text{and} \ |\lambda_{n-1}| \geq \frac{1}{2} \right\}.
\]
Since $\{x_{n} \mid n \in \mathbb{N} \}$ is linearly independent and $\lambda_{n-1} \neq 0$, it follows that $0 \notin K_{n-1}$. Moreover, $K_{n-1}$ is compact. Hence, there exists $\delta_{n-1} > 0$ such that $\| v \| > \delta_{n-1}$ for every $v \in K_{n-1}$. We then define 
\[
\gamma_{n} = \frac{1}{2\|x_{n}\|} \min \left\{ \delta_{n-1} , \gamma_{n-1} \| x_{n-1} \| , 2 \| x_{n} \| \right\}.
\]

Suppose, for contradiction, that 
\[
\sum_{n=1}^{\infty} \lambda_{n} \gamma_{n} x_{n} = 0
\]
for some sequence $(\lambda_{n}) \in \ell_{\infty} \setminus \{0\}$. Without loss of generality, we may assume that $\| (\lambda_{n}) \|_{\infty} = 1$. Then there exists $n_{0} \in \mathbb{N}$ such that $|\lambda_{n_{0}}| \geq \frac{1}{2}$. It follows that 
\[
\left\| \sum_{n=1}^{n_{0}} \lambda_{n} \gamma_{n} x_{n} \right\| = 
\left\| \sum_{n=n_{0}+1}^{\infty} (-\lambda_{n}) \gamma_{n} x_{n} \right\| 
\leq \sum_{n=n_{0}+1}^{\infty} |\gamma_{n}| \, \| x_{n} \| 
\leq \sum_{n=1}^{\infty} \frac{\delta_{n_{0}}}{2^{n}} 
= \delta_{n_{0}},
\]
which contradicts the choice of $\delta_{n_{0}}$.
\end{proof}

\begin{theorem}[Mycielski Theorem]\label{thm:mycielski}
Suppose that $X$ is a separable complete metric space without isolated points, and that for every $n \in \mathbb{N}$, the set $\mathcal{R}_n$ is residual in the product space $X^n$. Then there is a Mycielski set $\mathcal{K}$ in $X$ such that
\[
(y_1, y_2, \cdots, y_n) \in \mathcal{R}_n
\]
for each $n \in \mathbb{N}$ and any pairwise different $n$ points $y_1, y_2, \ldots, y_n$ in $\mathcal{K}$.

A set $\mathcal{K}$ is referred to as a Mycielski set if the intersection of $\mathcal{K}$ and any nonempty open set $U$ contains a Cantor set.
\end{theorem}

\section{Theorem \ref{thm:ext_f(T)}} \label{Sec:Proof_ext_f(T)}

\begin{proof}[\textit{Proof of Theorem \ref{thm:ext_f(T)}}]
Let $x \in HC(T)$, our goal is to construct a point $\xi \in HC(T)$, satisfying \[
\{ f(T)x \mid f \in H(\mathbb{C}) \} 
\subset \{ f(T)x \mid f \in H(\mathbb{C}) \} \oplus \{ g(T)\xi \mid g \in H(\mathbb{C}) \} 
\subset HC(T) \cup \{0\}
\]
and such that 
\[
\{ f(T)x \mid f \in H(\mathbb{C}) \} \cap \{ g(T)\xi \mid g \in H(\mathbb{C}) \} = \{0\}.
\]
Once $\xi$ is found the theorem is proved since $\{ g(T)\xi \mid g \in H(\mathbb{C}) \}$ has dimension $\mathfrak c$. 

Let us now construct such $\xi \in HC(T)$.
We want $\xi$ to visit all opens sets and while $\xi$ visits these sets we want $x$  to be visiting a neighborhood of zero.

Let $(x_{i}) \subset X_{0}$ be a dense sequence in $X$. And consider a sequence $(n_{k})$ such that $\| T^{n_{k}}x \| \leq 2^{-k}$ and define $k_{1} = 1$. Suppose now that $k_{1}, \dots , k_{i-1}$ are defined. Because we have a finite amount of variables to control we may choose $k_{i} > k_{i-1}$ satisfying: 

\begin{enumerate}[label=($\Lambda$\arabic*)]
\item $\| S^{m}x_{i} \| < 2^{-i}$ for all $m \geq n_{k_i} - n_{k_{i-1}}$;
\item $\| T^{m}x_{j} \| < 2^{-(i+j)}$ for all $m \geq n_{k_i} - n_{k_{i-1}}$ and $j \leq i$;
\end{enumerate}

In this case, by ($\Lambda$1), $\xi = \sum_{i=1}^{\infty} S^{n_{k_{i}}}x_{i}$ converges absolutely and satisfies $\| \xi \| \leq 1$. Let $f, g \in H(\mathbb C)$, since $f$ and $g$ can be written as a power series $g(T)$ and $f(T)$ commutes with $T^n$ for all $n \in \mathbb N$. 
Moreover,
\begin{align*}
& \left\| T^{n_{k_{j}}} [f(T)x + g(T)\xi] - g(T) x_{j} \right\| = \left\| T^{n_{k_{j}}} [f(T)x + g(T)(\sum_{i=1}^{\infty} S^{n_{k_{i}}}x_{i})] 
- g(T) x_{j} \right\|  \\ 
&\leq \|f(T)\| \cdot \| T^{n_{k_{j}}}x \| 
+ \| g(T) \| \sum_{i<j} \left\| T^{n_{k_{j}}-n_{k_{i}}} x_{i} \right\|  
+ \| g(T) \| \sum_{i>j} \left\| S^{n_{k_{i}}-n_{k_{j}}} x_{i} \right\| \\
&< \|f(T)\| \cdot 2^{-k_{j}} 
+ \| g(T) \| \sum_{i<j}2^{-(i+j)}  
+ \| g(T) \| \sum_{i>j} 2^{-i} \\
&< \|f(T)\| \cdot 2^{-k_{j}} + \| g(T) \| 2^{-j+1}.
\end{align*}

Now, if $g \neq 0$, then it is known that $g(T)$ has dense range (see \cite[Theorem 4.3.1]{aron2015lineability}). Thus, given a non-empty open set $U \subset X$, the set $g(T)^{-1}(U)$ is a non-empty open subset of $X$. Hence, there exists $B(y,\varepsilon) \subset g(T)^{-1}(U)$ and $x_{p} \in B(y,\frac{\varepsilon}{2})$, with $p>m$ and $m$ such that
\[
\| g(T) \| 2^{-m+1} + \|f(T)\| \cdot 2^{-k_{m}} < \frac{\varepsilon}{2}.
\]
Therefore, $T^{n_{k_{p}}} [f(T)x + g(T)\xi] \in U$ and $f(T)x + g(T)\xi$ is hypercyclic whenever $g \neq 0$.
Consequently, we also obtain that
\[
\{ f(T)x \mid f \in H(\mathbb{C}) \} \cap \{ g(T)\xi \mid g \in H(\mathbb{C}) \} = \{0\},
\]
otherwise there would exist $f,g \in H(\mathbb{C}) \setminus \{0\}$ such that $f(T)x+g(T)\xi = 0$, which is impossible since we showed that $f(T)x+g(T)\xi$ is hypercyclic whenever $g\neq 0$. Therefore,
\[
\{ f(T)x \mid f \in H(\mathbb{C}) \} 
\subset \{ f(T)x \mid f \in H(\mathbb{C}) \} \oplus \{ g(T)\xi \mid g \in H(\mathbb{C}) \} 
\subset HC(T) \cup \{0\}.
\]

\end{proof}

From the proof above and using induction we obtain the following result.

\begin{corollary}\label{cor:sum_entire_function}
For any operator $T$ satisfying Kitai's Criterion on a separable Banach space $X$, $n \in \mathbb N$ and $x \in HC(T)$, the set $\{ f(T)x \mid f \in H(\mathbb{C}) \}$, then there exist $\xi_1, \xi_2, \ldots, \xi_n$ such that 
$$\{ f(T)x \mid f \in H(\mathbb{C})\} \oplus \left( \bigoplus_{i=1}^{n} \{ f(T)\xi_i \mid f \in H(\mathbb{C})\} \right) \subset HC(T) \cup \{0\}.$$
\end{corollary}
Note that in this corollary the sum is a direct sum, that is, the sets intersect each other only at the origin.
	
\section{Theorems \ref{thm:G_delta_kitai} and \ref{thm:ext_Kitai}} \label{Sec:Proof_ext_Kitai}

We now proceed to the proofs of the remaining A's theorems.

\subsection{Proof of Theorem \ref{thm:G_delta_kitai}}\label{Subsec:Proof_G_delta}

\begin{proof}[Proof of Theorem \ref{thm:G_delta_kitai}] 

Since $T$ satisfies Kitai's criterion, it is known that $\bigoplus_{i=1}^{n} T$ is hypercyclic for all $n \in \mathbb N$, hence the set $HC(\bigoplus_{i=1}^{n} T) \subset X^n$ is a $G_\delta$ dense set. Therefore, the set $\mathcal{R}_n \subset X^n$ from item (i) can be taken to be $HC(\bigoplus_{i=1}^{n} T)$. And the set $\mathcal R_{\mathbb N}$ from item (ii) comes from applying Mycielski Theorem \ref{thm:mycielski}.

To conclude the proof we need only to check that if $(y_{1}, \dots , y_{n}) \in HC(\bigoplus_{i=1}^{n} T)$, then $\text{span}(\{ y_{1}, \dots , y_{n} \}) \subset HC(T) \cup \{ 0 \}$ and $(y_{i})$ can approximate $0$ uniformly. To see this, consider a linear combination $\lambda_{1}y_{1} + \dots + \lambda_{n} y_{n}$ and, without loss of generality, assume that $\lambda_{i} \neq 0$ for every $0 < i \leq n$. Then, given a non-empty open set $U$, $x \in U$ and $\varepsilon >0$ such that $B(x,\varepsilon) \subset U$, there exists a $k \in \mathbb{N}$ such that 
\[
(T^{k}y_{1}, \dots , T^{k}y_{n}) \in B\left( 0, \frac{\varepsilon}{2|\lambda_{1}|(n-1)} \right) \times \dots B \left( 0, \frac{\varepsilon}{2|\lambda_{n}|(n-1)} \right) \times B\left( \frac{x}{\lambda_{n}}, \frac{\varepsilon}{2|\lambda_{n}|} \right).
\]
So, 
\begin{align*}
\|T^{k}(\lambda_{1}y_{1} + \dots + \lambda_{n} y_{n}) - x\| 
&\leq |\lambda_{1}|\|T^{k}y_{1}\| + \dots + |\lambda_{n-1}|\|T^{k}y_{n-1}\| + |\lambda_{n}| \left\| T^{k}y_{n} - \frac{x}{\lambda_{n}} \right\| \\
&<|\lambda_{1}| \frac{\varepsilon}{2|\lambda_{1}|(n-1)} + \dots + |\lambda_{n-1}|\frac{\varepsilon}{2|\lambda_{n-1}|(n-1)} + |\lambda_{n}|\frac{\varepsilon}{2|\lambda_{n}|} \\
&< \varepsilon.
\end{align*}

Therefore, $T^{k}(\lambda_{1}y_{1} + \dots + \lambda_{n} y_{n}) \in B(x,\varepsilon) \subset U$ and we concluded that $\lambda_{1}y_{1} + \dots + \lambda_{n} y_{n} \in HC(T)$. Additionally, this also proves that $y_{1}, \dots , y_{n}$ are linearly independent. 
\end{proof}

\subsection{Proof of Theorem \ref{thm:ext_Kitai}}
\begin{proof}[Proof of Theorem \ref{thm:ext_Kitai}] 
Since $X$ is separable, we can enumerate a dense sequence of vectors $(x_{j})_{j} \subset X_{0}$. Furthermore, relabeling if necessary, we may assume that $(n_{k})$ is a subsequence such that, if $I_{k} = \{1 , \dots , k\}$, 

\[
T^{n_k} y_i \in B(0, 2^{-k}), \: \forall i \in \alpha \cap I_{k}.
\]
We aim to inductively construct a map $a: \mathbb{N} \times \mathbb{N} \rightarrow \{n_k\}_{k \in \mathbb{N}}$ and hypercyclic vectors $\xi_i = \sum_{l=1}^{\infty} S^{a(i,l)} x_l$, in such a way that the orbit of each $\xi_i$ approximate $x_l$ at times when the orbit of each $y_k$ is near the origin. For this, consider a bijection $\pi: \mathbb{N} \rightarrow \mathbb{N} \times \mathbb{N}$ of the form $\pi = (\pi_1, \pi_2)$ and define $a(\pi(1)) = n_1$.

Now, suppose that $a(\pi(1)), \dots, a(\pi(n-1))$ have been defined. We choose $a(\pi(n))$ satisfying:
\begin{enumerate}[label=($\Lambda$\arabic*)]
    \item $a(\pi(n-1)) < a(\pi(n))$;
    \item $\| S^{a(\pi(n))} x_{\pi_2(n)} \| < 2^{-(\pi_1(n) + \pi_2(n))}$;
    \item $\| T^{a(\pi(n)) - a(\pi(m))} x_{\pi_2(m)} \| < 2^{-(\pi_1(n) + \pi_1(m) + \pi_2(n) + \pi_2(m))}$ for all $m < n$;
    \item $\| S^{a(\pi(n)) - a(\pi(m))} x_{\pi_2(n)} \| < 2^{-(\pi_1(n) + \pi_1(m) + \pi_2(n) + \pi_2(m))}$ for all $m < n$.
\end{enumerate}

As required, $\xi_{i} = \sum_{l=1}^{\infty} S^{a(i,l)} x_{l}$ converges thanks to ($\Lambda$2) and $\|\xi_{i}\| < 2^{-i}$.
With a slight abuse of notation, we write $S^{-n} = T^n$, bearing in mind that $S$ is only a right inverse of $T$. Now, we can show that the orbit of $\xi_{i}$ approximate $x_{j}$ in the moment $a(i,j)$.
\begin{align*}
\left\| T^{a(i,j)} \xi_i - x_j \right\| 
&= \left\| \sum_{l \neq j} S^{a(i,l)-a(i,j)} x_l \right\| \\
&\leq \sum_{l \neq j} \left\| S^{a(i,l)-a(i,j)} x_l \right\| \\
&< \sum_{l\neq j} 2^{-(2i+j+l)} \quad \text{(by $\Lambda$3 and $\Lambda$4)} \\
&< 2^{-(2i+j)}.
\end{align*}

In addition, $\xi_{i}$ is close to $0$, whenever $\xi_{i'}$ is close to any $x_{j}$. More precisely, if $i' \neq i$ and $j \in \mathbb{N}$, then:
\begin{align*}
\left\| T^{a(i',j)} \xi_{i} \right\| 
&= \left\| \sum_{l=1}^{\infty} S^{a(i,l)-a(i',j)} x_l \right\| \\
&\leq \sum_{l=1}^{\infty} \left\| S^{a(i,l)-a(i',j)} x_l \right\| \\
&< \sum_{l=1}^{\infty} 2^{-(i+i'+j+l)} \\
&= 2^{-(i'+i+j)}.
\end{align*}

Finally, we can show that $[\{y_i \mid i \in \alpha \}] \oplus [ \xi_{i} \mid i \in \mathbb{N}] \subset HC(T) \cup \{0\}$. For that, consider a vector $x \in [\{y_i \mid i \in \alpha \}] \oplus [ \xi_{i} \mid i \in \mathbb{N}]$. If $x \in [\{y_i \mid i \in \alpha \}]$, by hypothesis, we have that $x \in HC(T) \cup \{0\}$. Therefore, we assume that $x = \beta_{1} y_{1} + \dots + \beta_{s} y_{s} + \lambda_{1} \xi_{1} + \dots + \lambda_{r} \xi_{r}$ such that $\lambda_{r} \neq 0$ and $\sup\{|\beta_{1}|, \dots , |\beta_{s}| , |\lambda_1|, \dots, |\lambda_r|\} = M$.

\begin{align*}
\left\| T^{a(r,j)}x - \lambda_r x_j \right\| 
&\leq \sum_{k=1}^{s} \left\| T^{a(r,j)} \beta_{k} y_{k} \right\| + \sum_{k=1}^{r-1} \left\| T^{a(r,j)} \lambda_k \xi_k \right\| + \left\| T^{a(r,j)} \lambda_r \xi_r - \lambda_r x_j \right\| \\
&<  \sum_{k=1}^{s} 2^{-\pi^{-1}(r,j)} |\beta_{k}| + \sum_{k=1}^{r-1} |\lambda_k| 2^{-(r+j+k)} + |\lambda_r| 2^{-(2r+j)} \\
&< M \left( s 2^{-\pi^{-1}(r,j)}+ 2^{-(r+j)} \right).
\end{align*}

In particular, that also proves that the set $\{y_{i} \mid i \in \alpha \} \cup \{ \xi_{i} \mid i \in \mathbb{N}\}$ is linearly independent, provided that $\{y_i \mid i \in \alpha \}$ is linearly independent.  
So, by Theorem \ref{thm:ell_infty-independence}, there are scalars $0 < \gamma_{i} \leq 1$ such that the operator 
\begin{align*}
A: \ell^\infty &\rightarrow X \\
(\lambda_i)_{i \in \mathbb{N}} &\mapsto \sum_{i \in \mathbb{N}} \lambda_i\gamma_{i} \xi_i,
\end{align*}
is an injection. It is easy to verify, using the triangle inequality, that $A$ is well defined and $\|A\| \leq 1$.

Now, let $z \in A(\ell^{\infty}) \oplus [\{y_i \mid i \in \alpha \}]$, so $z= \sum_{k=1}^{s}\beta_{k}y_{k} + \sum_{k=1}^{\infty}\lambda_{k}\gamma_{k} \xi_{k}$. If $\lambda_{k}=0$, for every $k>k_{0}$, for some $k_{0}$, then it has already been proven that $z$ is hypercyclic. Now, assume that $\lambda_{k} \neq 0$ for arbitrarily larges values of $k$ and set $M = \sup \left( \{\lambda_{k}\gamma_{k} \mid k\in\mathbb{N}\} \cup \{\beta_{k} \mid k\in\mathbb{N}\} \right)$, then 
\begin{align*}
\| T^{a(i,j)}z -  \lambda_{i}\gamma_{i} x_j\| & \leq 
\sum_{k=1}^{s} \left\| T^{a(i,j)} \beta_{k} y_k \right\| + 
\sum_{k \neq i} \left\| T^{a(i,j)} \lambda_k\gamma_{k} \xi_k \right\| + 
\left\| T^{a(i,j)} \lambda_{i}\gamma_{i} \xi_i - \lambda_{i}\gamma_{i} x_j \right\| \\
& < M( s 2^{- \pi^{-1}(i,j)} + 2^{-(i+j)} + 2^{-(2i+j)})
\end{align*}
Hence, given an open ball $B(w,\varepsilon)$, we can choose a sufficiently large $i$ such that $\lambda_{i}\neq 0$ and $M( s 2^{- \pi^{-1}(i,j)} + 2^{-(i+j)} + 2^{-(2i+j)}) < \frac{\varepsilon}{2}$. Then, we can choose $j$ such that $\lambda_{s+i}\gamma_{s +i} x_{j} \in B(w,\frac{\varepsilon}{2})$. Thus, $T^{a(i,j)}z \in B(w,\varepsilon)$.  

Therefore, $W = A(\ell^{\infty}) \oplus [\{y_i \mid i \in \alpha \}]$ satisfies the desired properties: 
\[
[\{y_i \mid i \in \alpha \}] \subset W \subset HC(T) \; \text{and} \; \dim(W) \geq \dim(\ell^{\infty}) = \mathfrak{c}. 
\]
\end{proof}

\section{Theorem \ref{thm:ext_f(T)_Freq}} \label{Sec:Proof_ext_f(T)_Freq}

We start with an auxiliary result, but interesting on its own.
\begin{proposition}
If $x$ is frequently hypercyclic, then $f(T)x$ is frequently hypercyclic for every entire function $f$. 
\end{proposition}

\begin{proof}
Let $U \subset X$ be a non-empty open set. Since $f(T)$ is continuous with dense range (see \cite[Theorem 4.3.1]{aron2015lineability}), $f(T)^{-1}(U)$ is a non-empty open set. Hence, $x$ has positive frequency in $f(T)^{-1}(U)$. Therefore, since $f(T)$ commutes with $T$, it follows that $f(T)x$ has positive frequency in $U$.
\end{proof}

\subsection{Proof of Theorem \ref{thm:ext_f(T)_Freq}}
\begin{proof}[Proof of Theorem \ref{thm:ext_f(T)_Freq}]
Once again, our goal is to construct a point $\xi \in FHC(T)$ that 
visits all open sets at times $n_{k}$, while $x$ is close to zero.
Let $(x_{j}) \subset X_{0}$ be a dense sequence. By \ref{CHF(1)} and \ref{CHF(2)}, there are $N_{k}$ such that for any $j \leq k$ and any finite set $F \subset \{ N_{k} , N_{k}+1 , N_{k}+2, \dots \}$ we have that
\begin{equation}
\left\| \sum_{n \in F} T^{n} x_{j} \right\| \leq \frac{1}{k2^{k}} ,
\end{equation}
\begin{equation}
\left\| \sum_{n \in F} S^{n} x_{j} \right\| \leq \frac{1}{k2^{k}} .
\end{equation}

Now let $A(i,\nu)$ be as given by Lemma \ref{lem:positive_density} and set
\[
A = \bigcup_{\nu=1}^{\infty} A(\nu,N_{\nu}) \text{ and } z_{k} = x_{j} \text{ if } k \in A(j,N_{j}).
\]
We define
\[
\xi = \sum_{k \in A} S^{n_{k}}z_{k}.
\]
To see that this series converges unconditionally, given $\varepsilon > 0$, choose $\kappa \in \mathbb{N}$ such that $\frac{1}{2^{\kappa-1}} < \varepsilon$. Then, for any finite set $F \subset \{ N_{\kappa} , N_{\kappa}+1 , N_{\kappa}+2 , \dots \}$,

\begin{align*}
\left\| \sum_{\substack{k \in A \\ n_{k} \in F}} S^{n_{k}} z_{k} \right\|
&\leq \sum_{j=1}^{\kappa} \left\| \sum_{\substack{k \in A(j,N_{j}) \\ n_{k} \in F}} S^{n_{k}} x_{j} \right\|
+ \sum_{j=\kappa + 1}^{\infty} \left\| \sum_{\substack{k \in A(j,N_{j}) \\ n_{k} \in F}} S^{n_{k}} x_{j} \right\| \\
&\leq \sum_{j=1}^{\kappa} \frac{1}{\kappa 2^{\kappa}} + \sum_{j=\kappa + 1}^{\infty} \frac{1}{j 2^{j}} \\
&< \frac{1}{2^{\kappa}} + \sum_{j=\kappa + 1}^{\infty} \frac{1}{2^{j}} \\
& = \frac{1}{2^{\kappa - 1}} < \varepsilon
\end{align*}

Now, we can show that, for any $f,g \in H(\mathbb{C})$ with $g \neq 0$, $f(T)x + g(T)\xi$ is frequently hypercyclic. Then, given any open set $U \subset X$, take some $U' \subset U$ and $\varepsilon > 0$ such that for any $u \in U'$, $B(u,\varepsilon) \subset U$. Choose some $x_{j} \in g(T)^{-1}(U')$ such that $2^{-j} < \frac{\varepsilon}{2\|g(T)\|}$. If $\kappa \in A(j,N_{j})$ and $\kappa > k_{0}$, where $k_{0}$ is such that $\|T^{n_{k}} x \| < \frac{\varepsilon}{2}$, for all $k>k_{0}$, then $|n_{\kappa} - n_{k}| \geq |\kappa - k| \geq N_{j}$, for any $k \in A \setminus \{\kappa\}$ and

\begin{align*}
& \left\| T^{n_{\kappa}} \left( f(T)x + g(T)\xi \right) - g(T)x_{j} \right\| \\
\leq & \| f(T) \| \left\| T^{n_{\kappa}} x \right\| 
+ \| g(T) \| \left( \left\| \sum_{\substack{k \in A \\ k < \kappa}} T^{n_{\kappa} - n_{k}} z_{k} \right\| 
+ \left\| \sum_{\substack{k \in A \\ k > \kappa}} S^{n_{k}-n_{\kappa}} z_{k} \right\| \right) \\
< & \| f(T) \| \left\| T^{n_{\kappa}} x \right\| + \frac{\|g(T)\|}{2^{j}} \\
< & \varepsilon.
\end{align*}

Therefore, $T^{n_{\kappa}} \left( f(T)x + g(T)\xi \right) \in U$ and, by Lemma \ref{lem:positive_density_subsequence}, 
\[
\text{freq}(\{ n_{\kappa} \mid \kappa \in A(j,N_{j}) \text{ and } \kappa > k_{0}\}) \geq \text{freq}(A(j,N_{j})) \cdot \text{freq}((n_{k})_{k\in\mathbb{N}}) >0.
\]
Moreover, that also shows that $\{f(T)x \mid f \in H(\mathbb{C})\} \cap \{g(T)\xi \mid g \in H(\mathbb{C})\} = \{0\}$, otherwise there would exist $g \neq 0$ such that $f(T)x + g(T)\xi = 0$, which cannot be frequently hypercyclic.
\end{proof}

Additionally, the same argument can be used to show the following corollary.

\begin{corollary}\label{cor:sum_entire_function_freq}
For any operator $T$ satisfying the Frequently Hypercyclicity Criterion on a separable Banach space $X$, $n \in \mathbb N$ and $x \in HC(T)$, such that $T^{n_{k}}x \to 0$, for some sequence $(n_{k})$ with positive lower density, then there exist $\xi_1, \xi_2, \ldots, \xi_n$ such that 
$$\{ f(T)x \mid f \in H(\mathbb{C})\} \oplus \left( \bigoplus_{i=1}^{n} \{ f(T)\xi_i \mid f \in H(\mathbb{C})\} \right) \subset HC(T) \cup \{0\}.$$
\end{corollary}

\section{Theorem \ref{thm:ext_CHF}} \label{Sec:Proof_ext_CHF}

\subsection{Proof of Theorem \ref{thm:ext_CHF}}
\begin{proof}[Proof of Theorem \ref{thm:ext_CHF}]
Since $X$ is separable, we can enumerate a dense sequence of vectors $x_{j} \in X_{0}$, for all $j \in \mathbb{N}$. 
By \ref{CHF(1)} and \ref{CHF(2)}, there are $N_{k}$ such that for any $j \leq k$ and any finite set $F \subset \{ N_{k} , N_{k}+1 , N_{k}+2, \dots \}$ we have that
\begin{equation}
\left\| \sum_{n \in F} T^{n} x_{j} \right\| \leq \frac{1}{k2^{k}} ,
\end{equation}
\begin{equation}
\left\| \sum_{n \in F} S^{n} x_{j} \right\| \leq \frac{1}{k2^{k}} .
\end{equation}

Now let $A(i,j,\nu)$ be as given by Lemma \ref{lem:positive_density_alt} and set
\[
A_{l} = \bigcup_{k=1}^{\infty} A(l,k,N_{k}) \text{ and } z_{k} = x_{j} \text{ if } k \in A(l,j,N_{j}).
\]

We define
\[
\xi_{l} = \frac{1}{2^{l}} \sum_{k \in A_{l}} S^{n_{k}}z_{k}, \text{ for each } l \in \mathbb{N}.
\]
To see that this series converges unconditionally, fix $l \in \mathbb{N}$ and, given $\varepsilon > 0$, choose $\kappa \in \mathbb{N}$ such that $\frac{1}{2^{\kappa-1}} < \varepsilon$. Then, for any finite set $F \subset \{ N_{\kappa} , N_{\kappa}+1 , N_{\kappa}+2 , \dots \}$,

\begin{align*}
\left\| \sum_{\substack{k \in A_{l} \\ n_{k} \in F}} S^{n_{k}} z_{k} \right\|
&\leq \sum_{j=1}^{\kappa} \left\| \sum_{\substack{k \in A(l,j,N_{j}) \\ n_{k} \in F}} S^{n_{k}} x_{j} \right\|
+ \sum_{j=\kappa + 1}^{\infty} \left\| \sum_{\substack{k \in A(l,j,N_{j}) \\ n_{k} \in F}} S^{n_{k}} x_{j} \right\| \\
&\leq \sum_{j=1}^{\kappa} \frac{1}{\kappa 2^{\kappa}} + \sum_{j=\kappa + 1}^{\infty} \frac{1}{j 2^{j}} \\
&< \frac{1}{2^{\kappa}} + \sum_{j=\kappa + 1}^{\infty} \frac{1}{2^{j}} \\
& = \frac{1}{2^{\kappa - 1}} < \varepsilon
\end{align*}

Now, let $r \in A(\tilde{l},j,N_{j})$. Then, 
\[
2^{l} \| T^{n_{r}}\xi_{l} \| 
\leq \left\| \sum_{\substack{k \in A_{l} \\ k < r}} T^{n_{r} - n_{k}} z_{k} \right\|
+ \left\| \sum_{\substack{k \in A_{l} \\ k > r}} S^{n_{k}-n_{r}} z_{k} \right\|
\]

In addition, if $r \in A(\tilde{l},j,N_{j})$ and $k \in A(l,i,N_{i})$, with $\tilde{l} \neq l$, then $|n_{r} - n_{k}| \geq |r-k| \geq N_{j}+N_{i}$. So, in the first sum, we have that 

\begin{align*}
\left\| \sum_{\substack{k \in A_{l} \\ k < r}} T^{n_{r} - n_{k}} z_{k} \right\| 
&\leq \sum_{i=1}^{j} \left\| \sum_{\substack{k \in A(l,i,N_{i}) \\ k < r}} T^{n_{r} - n_{k}} x_{i} \right\| 
+ \sum_{i=j + 1}^{\infty} \left\| \sum_{\substack{k \in A(l,i,N_{i}) \\ k < r}} T^{n_{r} - n_{k}} x_{i} \right\| \\
&\leq \sum_{i=1}^{j} \frac{1}{j 2^{j}} + \sum_{i=j + 1}^{\infty} \frac{1}{i 2^{i}} \\
&< \frac{1}{2^{j - 1}}.
\end{align*}

Consequently, repeating the same argument for the second term, we have that
\[
\left\| T^{n_{r}} \xi_{l} \right\| < \frac{1}{2^{l+j-2}} \text{, whenever } r \in A(\tilde{l},j,N_{j}) \text{ with }\tilde{l} \neq l.
\]

Moreover, we can show that $( [\{y_i \mid i \in \alpha \}] \oplus [ \xi_{i} \mid i \in \mathbb{N}] ) \setminus \{0\} \subset FHC(T)$. For that, consider a vector $x \in [\{y_i \mid i \in \alpha \}] \oplus [ \xi_{i} \mid i \in \mathbb{N}]$. If $x \in [\{y_i \mid i \in \alpha \}]$, by hypothesis, we have that $x \in FHC(T) \cup \{0\}$. Therefore, we assume that $x = \beta_{1} y_{1} + \dots + \beta_{s} y_{s} + \lambda_{1} \xi_{1} + \dots + \lambda_{l-1} \xi_{l-1} + \xi_{l}$ such that $\lambda_{r} \neq 0$ and $\sup\{|\beta_{1}|, \dots , |\beta_{s}| , |\lambda_1|, \dots, |\lambda_l|\} = M$. Given $j \in \mathbb{N}$, for any $r \in A(l,j,N_{j})$, 
\begin{align*}
\left\| T^{n_{r}} x - x_{j} \right\| 
&\leq M\sum_{i=1}^{s} \left\|T^{n_{r}}y_{i} \right\| 
+ M\sum_{i=1}^{l-1} \left\| T^{n_{r}} \xi_{i} \right\| 
+ \left\| \sum_{\substack{k \in A_{l} \\ k < r}} T^{n_{r} - n_{k}} z_{k} \right\| 
+ \left\| \sum_{\substack{k \in A_{l} \\ k > r}} S^{n_{k}-n_{r}} z_{k} \right\| \\
&< M\sum_{i=1}^{s} \left\|T^{n_{r}}y_{i} \right\| 
+ M\frac{l-1}{2^{j-2}} 
+ \frac{1}{2^{j-1}} 
+ \frac{1}{2^{j-1}} \\
&= M\sum_{i=1}^{s} \left\|T^{n_{r}}y_{i} \right\| + \frac{M(l-1)+1}{2^{j-2}}
\end{align*}

Since $T^{n_k} y_i \rightarrow 0$ uniformly, $\left\| T^{n_{r}} x - x_{j} \right\| < \varepsilon$ for sufficiently large values of $r$. So, repeating the same argument from Theorem \ref{thm:ext_Kitai}, we define $A:\ell^{\infty} \rightarrow X$ as before and let $z \in A(\ell^{\infty}) \oplus [\{y_i \mid i \in \alpha \}]$. Then, given $j \in \mathbb{N}$, for any $r \in A(l,j,N_{j})$, 
\begin{align*}
\left\| T^{n_{r}} z - x_{j} \right\| 
&\leq M\sum_{i=1}^{s} \left\|T^{n_{r}}y_{i} \right\| 
+ M\sum_{i \neq j} \left\| T^{n_{r}} \xi_{i} \right\| 
+ \left\| \sum_{\substack{k \in A_{l} \\ k < r}} T^{n_{r} - n_{k}} z_{k} \right\| 
+ \left\| \sum_{\substack{k \in A_{l} \\ k > r}} S^{n_{k}-n_{r}} z_{k} \right\| \\
&< M\sum_{i=1}^{s} \left\|T^{n_{r}}y_{i} \right\| 
+ M\sum_{i \neq j} \frac{1}{2^{l+j-2}} 
+ \frac{1}{2^{j-1}} 
+ \frac{1}{2^{j-1}} \\
&< M\sum_{i=1}^{s} \left\|T^{n_{r}}y_{i} \right\| + \frac{M+1}{2^{j-2}}.
\end{align*} 

Therefore, we obtain the desired extension by $[\{y_i \mid i \in \alpha \}] \subset A(\ell^{\infty}) \oplus [\{y_i \mid i \in \alpha \}] \subset FHC(T)$ and $\text{dim}(A(\ell^{\infty})) = \text{dim}(\ell^{\infty}) = \mathfrak{c}$.
\end{proof}

\section*{Acknowledgements}
F.C.S. was partially supported by the Coordenação de Aperfeiçoamento de Pessoal de Nível Superior do Brasil (CAPES) (grant 682852/2022-00). G.R. was supported by CAPES — Coordenação de Aperfeiçoamento de Pessoal de Nível Superior (Brazil) — through a postdoctoral fellowship at IMECC, Universidade Estadual de Campinas (PIPD/CAPES; Finance Code 001). R.V. was supported by the São Paulo Research Foundation (FAPESP), grant 24/15612-6, and  by Conselho Nacional de Desenvolvimento Científico e Tecnológico (CNPq), grant 314978/2023-2. 

\nocite{aron2015lineability}
\nocite{favaro2024lineability}
\nocite{bayart2006frequently}
\nocite{bayart2009dynamics}
\nocite{bes1999hereditarily}
\nocite{grosseerdmann2011linear}

\bibliographystyle{plain}
\bibliography{Ref}

\end{document}